\documentclass[a4paper,12pt]{article}

\usepackage{amsmath,amssymb,amsthm}
\usepackage{mathtools}
\usepackage{color}
\usepackage[mathletters]{ucs}
\usepackage[utf8x]{inputenc}
\usepackage{hyperref}
\usepackage{enumitem}

\hypersetup{
colorlinks=false,
pdfborder={0 0 0},
}

\newtheorem{theorem}{Theorem}
\newtheorem{lemma}{Lemma}
\newtheorem{corollary}{Corollary}

\newtheorem{conjecture}{Conjecture}
\newtheorem{question}{Question}

\newtheorem{remark}{Remark}

\newcommand{\pr}{{\mathbb{P}}}
\newcommand{\bern}{{X}}
\newcommand{\Z}{{\mathbb{Z}}}

\newcommand{\R}{{\mathbb{R}}}
\newcommand{\N}{{\mathbb{N}}}

\newcommand{\E}{{\mathbb{\, E\,}}}

\newcommand{\Var}{{\text{Var\,}}}
\def\={\stackrel {\rm def}  {=}}
\def\di={\overset{\text{${\mathcal{D} }$}} =}

\newenvironment{proofof}[1]{%
\noindent {\bf Proof of #1}}%
{\hspace*{\fill}$\blacksquare$}

\title{On Littlewood--Offord problem for arbitrary distributions (full version)}
\author{T. Ju\v skevi\v cius\thanks{This project has received funding from European Social Fund (project No 09.3.3-LMT-K-712-02-0151) under grant agreement with the Research Council of Lithuania (LMTLT).},\, V. Kurauskas \\ \\ Vilnius University}

\begin{document}

\maketitle
\begin{abstract}
    Let $X_1,\ldots,X_n$ be independent identically distributed discrete random vectors in $\mathbb{R}^d$.
    We consider upper bounds on $\sup_x \mathbb{P}(a_1X_1+\cdots+a_nX_n=x)$ under various restrictions on $X_i$ and weights $a_i$.
    When $\pr(X_i=\pm 1) = \frac 1 2$, this corresponds to the classical Littlewood--Offord problem. 
    We prove that in general for identically distributed random vectors and even values of $n$ the optimal choice for $(a_i)$ is $a_i=1$ for $i\leq \frac{n}{2}$ and $a_i=-1$ for $i > \frac {n} 2$, regardless of the distribution of $X_1$. 
    Applying these results to Bernoulli random variables answers a recent question of Fox, Kwan and Sauermann.
    
    Finally, we provide sharp bounds for concentration probabilities of sums of random vectors under the condition $\sup_{x}\mathbb{P}(X_i=x)\leq \alpha$, where it turns out that the worst case scenario is provided by distributions on an arithmetic progression that are in some sense as close to the uniform distribution as possible.

    Unlike much of the literature on the subject we use neither methods of harmonic analysis nor those from extremal combinatorics.
\end{abstract}

\section{Introduction}

Let $X_{1},\ldots, X_{n}$ be independent identically distributed (iid) discrete random vectors in $\mathbb{R}^d$. In this paper we shall be interested in bounding probabilities $\mathbb{P}(a_1X_1+\cdots+a_nX_n=x)$ under various assumptions on the weights $a_i$ and the distributions of $X_i$. 

A special case when $X_i$ are Rademacher random variables, i.e., $\mathbb{P}(X_i=\pm 1)=\frac{1}{2}$, is known as the Littlewood--Offord problem; the classical result of Erd\H{o}s \cite{ELO,Sperner} is that for non-zero real weights $a_i$ we have
\begin{equation}\label{LO}
    \mathbb{P}(a_1X_{1}+\cdots+a_nX_{n}=x)\leq \mathbb{P}(X_{1}+\cdots+X_{n}\in \{0,1\})=\frac{\binom{n}{\lfloor\frac{n}{2}\rfloor}}{2^n} = {\sqrt{ \frac 2 {\pi n}}} + O(n^{-\frac 3 2}).
\end{equation}
Kleitman \cite{KLO} proved that the latter result remains true for Rademacher random variables and $a_i\in \R^d$. That is, linear combinations with equal weights exhibit the worst case behaviour.
%Assuming that $a_i$ are non-zero distinct real numbers Stanley \cite{Stanley} showed that
%$$\mathbb{P}(a_1X_{1}+\cdots+a_nX_{n}=x)\leq \mathbb{P}(X_{1}+2X_{2}+\cdots+nX_{n}\in \{0,1\})\approx \frac{1}{n^{\frac{3}{2}}},
%which was conjectured by Erd\H{o}s and Moser in \cite{em}. 
%
The same problem with restrictions on the arithmetical structure of the weights $a_i$ was considered in \cite{em, halasz, Stanley}. Fairly recently the Littlewood--Offord problem was considered in certain matrix groups in \cite{Tiep} and sharp results for arbitrary groups were obtained in \cite{JS}.

The first goal of the present work is to extend these problems to random variables with an arbitrary distribution in~$\R^d$. It turns out that for even values of $n$ there is a unique choice of weights $a_i$ that is optimal for arbitrary distributions. The case of odd values of $n$ is discussed later on in the paper. Let us state a result that is essential for all other results in the paper. 
\begin{lemma}\label{l1}
Let $n$ be a positive even integer and let $X_1,\ldots,X_n$ be independent discrete random vectors in $\mathbb{R}^d$. Then there is $j\in \{ 1,2,\ldots, n\}$ such that for all $x\in \mathbb{R}^d$
\begin{equation*}
    \mathbb{P}( X_{1}+\cdots+X_{n} = x)\leq \mathbb{P}( Y_{1}-Y_{2}+\cdots+Y_{n-1}-Y_{n}= 0)
\end{equation*}
where $Y_1, \dots, Y_n$ are iid copies of $X_j$.  
    The inequality is strict unless $\sum X_i-x$ and $\sum (-1)^{i+1} Y_i$ have the same distribution.
\end{lemma}
Intuitively, this lemma says that 
%if we make a an even number of independent random steps and want to hit a particular value $x$ with maximum probability, then this probability is not greater than going back and forth according to the distribution of one of these steps. 
the probability for a random walk with an even number of steps from some class of distributions to hit a particular value $x$ is never greater than the probability to hit the origin by repeatedly going back and forth according to some specific distribution from the class.
%This has the following straightforward consequences for linear combinations of independent identically distributed random variables.
%Interestingly, the 
Its proof 
%of Lemma \ref{l1} 
%does not use neither harmonic analysis, which is the bread and butter of the subject, nor does it use extremal combinatorics as Erd\H{o}s did in his proof of the optimal bound in \eqref{LO}. I
is very simple:
it merely uses multiple applications of the comparison between the arithmetic and geometric means.

Two straightforward consequences of Lemma~\ref{l1} are
\begin{corollary} For even $n$ and any $x \in \mathbb{R}^d$ we have
    $$\mathbb{P}( a_1X_{1}+\cdots+a_{n}X_{n} = x)\leq \mathbb{P}( X_{1}-X_{2}+\cdots+X_{n-1}-X_{n} = 0)$$
    \begin{enumerate}[label={(\alph*)},itemindent=1em]
        \item for iid discrete real random variables $X_i$ and any non-zero $a_i\in \mathbb{R}^d$; and \label{cor1}
        \item for iid discrete random vectors $X_i$ in $\mathbb{R}^d$ and any non-zero $a_i\in \mathbb{R}$. \label{cor2}
    \end{enumerate}
\end{corollary}
%%Its straightforward consequences are
%%\begin{corollary}\label{cor1}
%%For iid real random variables $X_i$ and any non-zero $a_i\in \mathbb{R}^d$ we have
%%$$\mathbb{P}( a_1X_{1}+\cdots+a_{2k}X_{2k} = x)\leq \mathbb{P}( X_{1}-X_{2}+\cdots+X_{2k-1}-X_{2k} = 0).$$
%%\end{corollary}
%%\begin{corollary}\label{cor2}
%%For iid real random vectors $X_i$ in $\mathbb{R}^d$ and any non-zero $a_i\in \mathbb{R}$  we have
%%$$\mathbb{P}( a_1X_{1}+\cdots+a_{2k}X_{2k} = x)\leq \mathbb{P}( X_{1}-X_{2}+\cdots+X_{2k-1}-X_{2k} = 0).$$
%%\end{corollary}
In other words, for even values of $n$ the worst case scenario in the latter two situations 
is provided by the balanced collection of $\pm 1$s,
%occurs taking half of the weights $1$ and half $-1$,
regardless of the distribution of the random variables $X_{i}$. Therefore we shall refer to Lemma~\ref{l1} as the ``balancing lemma''. 
\begin{remark}\label{rmk.1} % All the bounds obtained for an even number of random variables also give a bound for odd values in the following way. Since the maximum probability of $X_{1}+\cdots+X_{2k+1}$ is a convex combination of the probabilities of $X_{1}+\cdots+X_{2k}$, it is therefore bounded by the largest probability of the latter.
    Bounds for even values of $n$ also give bounds for odd values since by conditioning on $X_{n+1}$ (or by monotonicity of the L\'{e}vy concentration function)
    for any $n \ge 1$
    \[
        \max_{x} \pr(X_1 + \dots + X_{n+1}=x) \le \max_{x} \pr(X_1+\dots+X_{n}=x).
    \]
\end{remark}
The second part of our work has a bit different flavour. Instead of linear combinations of random vectors with given distributions we consider sums of independent random vectors $X_i$ in $\mathbb{R}^d$ such that no $X_i$ takes a particular value with too large a probability. For $\alpha\in (0, 1)$ we shall denote by $U^{\alpha}$ a random variable such that $\mathbb{P}(U^{\alpha}=l)=\alpha$ for $l=0,1,\ldots, \lfloor \frac{1}{\alpha} \rfloor-1$ and $\mathbb{P}(U^{\alpha}=\lfloor \frac{1}{\alpha} \rfloor)=1-\mathbb{P}(U^{\alpha}\in \{ 0,\ldots, \lfloor \frac{1}{\alpha} \rfloor-1\})$. For $\alpha=\frac{1}{k}$ with $k\in \N$ this random variable has the uniform distribution on $ \{ 0,\ldots,k-1\}$ and for $\alpha\in [\frac{1}{2}, 1)$ it has the Bernoulli distribution with parameter $1-\alpha$. We then establish the following inequality.
\begin{theorem}\label{thm2}
    Let $n$ be a positive even integer and let $\alpha \in (0,1)$.
    Let $X_1,\ldots, X_{n}$ be independent random vectors in $\R^{d}$ such that for all $i\in\{1,\dots,n\}$ we have
$$\sup_{x\in \R^d}\mathbb{P}(X_i=x)\leq \alpha.$$  
Then
$$\mathbb{P}(X_{1}+\cdots+X_{n} = x)\leq \mathbb{P}( U_{1}^{\alpha}-U_{2}^{\alpha}+\cdots+U_{n-1}^{\alpha}-U_{n}^{\alpha} = 0),$$
where the random variables $U_{i}^{\alpha}$ are iid copies of $U^\alpha$.
\end{theorem}
Note that the latter inequality is optimal as the random variables $U_{i}^{\alpha}$ satisfy the condition of the theorem.
%\\
%The latter 
The result for $\alpha=\frac{1}{k}$ with $k \in \N$ was established by Rogozin~\cite{rog} and also follows from the results of Leader and Radcliffe~\cite{LR}. Bounds for arbitrary $\alpha$ were obtained by Ushakov~\cite{Ushakov}, but they were not optimal when $\alpha < \frac 1 2$. We postpone the detailed discussion regarding a more complete history of this problem to Section 3.

\begin{corollary}\label{constant}
In the setting of Theorem~\ref{thm2} for all $n$ (even or odd) we get
\begin{align}
    &\mathbb{P}(X_{1}+\cdots+X_{n} = x) \le \left(2 \pi n \Var(U^\alpha)\right)^{-\frac 1 2} (1+o(1)),
    %& \mathbb{P}( U_{1}^{\alpha}-U_{2}^{\alpha}+\cdots+U_{2k-1}^{\alpha}-U_{2k}^{\alpha} = 0) =  
    %c_\alpha n^{-1/2} (1 + o(1)) 
    \quad \mbox {\emph{where}} \\
    & \Var(U^\alpha) = \frac 1 {12} \lfloor \alpha^{-1} \rfloor  (\lfloor \alpha^{-1} \rfloor  + 1) \alpha
    (2 + 4 \lfloor \alpha^{-1} \rfloor - 3 \alpha \lfloor \alpha^{-1} \rfloor - 3 \alpha \lfloor \alpha^{-1} \rfloor^2)  
    %((4-3 \alpha)  \lfloor \alpha^{-1} \rfloor  + 2- 3  \alpha  \lfloor \alpha^{-1} \rfloor^2)
    . \nonumber
\end{align}
\end{corollary}
When $\alpha^{-1}$ is integer, this simplifies to $\Var(U^\alpha) = \frac {1-\alpha^2} {12 \alpha^2}$.
%\frac{\sqrt 6 \alpha} {\sqrt{(1 - \alpha^2) \pi}}$.

Recently Fox, Kwan and Sauermann \cite{FKS} have posed the following question (we rephrase it slightly).
\begin{question}\label{qfox}
Let $a_1, \dots, a_n$ be non-zero real numbers and let $X_1, \dots, X_n$ be independent Bernoulli random variables with parameter $0<p\leq \frac{1}{2}$. What upper bounds (in terms of $n$ and $p$) can we give on the maximum point probability
$$\max_{x\in \R}\pr(a_1X_1+\cdots+a_nX_n=x)?$$  
\end{question}
Taking Bernoulli random variables in either Corollary~1\ref{cor1} or 1\ref{cor2} we obtain the optimal bound for the probability in question for even values of $n$ under more general conditions. Alternatively, it is a special case of Theorem~\ref{thm2} applied with $\alpha=1-p$. % establishes the same bound under more general circumstances - when the random variables have bounded largest probability and $n$ is even. 
%The situation for odd values of $n$ is more complex and it ilustrates the general difficulty in this case for arbitrary distributions and is the motivation of
%Obtaining exact answer to Question~\ref{qfox} 
The situation for odd $n$ seems to be much more involved.
%Nevertheless, we establish the optimal 
We were only able to prove that optimal $a_i$ must be $\pm 1$ when $n$ is large enough (see Section 4) and get some partial results illustrating why this case is more difficult.

For the Bernoulli case, or $\alpha \ge \frac 1 2$ in Theorem~\ref{thm2}, 
Ushakov's paper \cite{Ushakov}, communicated by Prokhorov in the early 80s,
already contains the asymptotically sharp bound 
$$\pr(X_1+\cdots+X_n=x) \leq (2 \pi np(1-p))^{-\frac 1 2}(1+O(n^{-\frac 1 2})).$$ 
\begin{remark} A solution to Question~\ref{qfox} has been very recently and independently obtained by Singhal \cite{Singhal} using different methods with stronger results than ours in the case when $n$ is odd.
\end{remark}
The paper is organized as follows. Lemma~\ref{l1} and Corollary~1 are proved in Section 2. Theorem~\ref{thm2} is proved and the history of the problem is discussed in Section 3. Section 4 is devoted
%to the discussion about the conjecture of Fox, Kwan and Sauermann and
to the Bernoulli case and results for odd values of $n$.
Finally, we present and discuss some open problems in Section 5.

\section{Proof of the balancing lemma}

We shall use the notation $X\sim Y$ to denote the fact that the random vectors $X$ and $Y$ have the same distribution. We shall now proceed with an elementary proof of the balancing lemma, which
also works for random summands taking values in a countable subset of an Abelian group.

\medskip

\begin{proofof}{Lemma~\ref{l1}}
    It is enough to prove the lemma for $x=0$, otherwise we can redefine $X_1$ as $X_1-x$.

    Let us split the sum into two halves: 
    \[
        S = \sum_{i=1}^{\frac n 2} X_i \quad \mbox{and} \quad T=\sum_{i={\frac n 2} + 1}^{n} X_i.
    \]
    Let $S'$ and $T'$ be independent copies of $S$ and $T$ respectively.
    By the inequality of arithmetic and geometric means
    \begin{align}
        \pr(S + T = 0) &= \sum_x \pr(S=x) \pr(-T=x) \le \sum_x \frac {\pr(S=x)^2 + \pr(-T=x)^2} 2 \nonumber
        \\ &=\frac 1 2 \sum_x \pr(S=x)^2 + \frac 1 2 \sum_x \pr(-T=x)^2 \nonumber
        \\ &\leq \max \left\{\sum_x \pr(S=x)^2, \sum_x \pr(T=x)^2\right\}\nonumber 
        \\ &= \max \left\{ \pr(S - S' = 0), \pr(T - T' = 0) \right\}. \label{eq.agm}
    \end{align}
    Note that for non-negative $p$ and $q$, $p q \leq \frac {p^2 + q^2} 2$ and
    we have an equality if and only if $p=q$. Therefore (\ref{eq.agm}) is equality if and only if $T \sim -S$.

    We will say that random vectors $X$ and $Y$ have the same \emph{type} if either $Y$ or $-Y$ has the same distribution as $X$.
    %Given a sequence of random variables $(Z_i, i \in \{1,\dots,2n\})$ consider equivalence classes (\emph{types}) defined by the equivalence relation $Z_i \sim Z_j \, \mathrm{or} \, Z_i \sim -Z_j$.
    Consider the different equivalence classes (types) of $\{X_1, \dots, X_{n}\}$ defined by the above equivalence relation. Note that $S-S'$ is a sum of $n$ independent random vectors whose terms preserve the types of $S$,
    furthermore, each term in $S$ is matched by a term with an opposite sign in $S'$ and similarly for $T-T'$.
    Thus if all the random variables have the same type, the proof follows by (\ref{eq.agm}).

    If there are more than two types, let $\mathcal{X}_1 = \{X_{i_1}, \dots, X_{i_k}\}$ and $\mathcal{X}_2 = \{X_{j_1}, \dots, X_{j_l}\}$ be different classes other than the largest equivalence class (break ties arbitrarily). Clearly $k \leq \frac n 2$ and $l \leq \frac n 2$. Rearrange the variables so that all the variables in $\mathcal{X}_1$ are in $S$ and all the variables in $\mathcal{X}_2$ are in $T$. Applying (\ref{eq.agm}) yields a new sequence
    of random variables $X_1', \dots, X_{n}'$ with $X_{2k}' \sim -X_{2k-1}'$, $k \in \{1, \dots, \frac n 2\}$ which has at least one less type and
    \[
        \pr(X_1 + \dots + X_{n} = 0) \leq \pr(X_1' + \dots + X_{n}'=0).
    \]
    By repeating this argument at most $n$ times, we reduce the number of types to one or two.
    It remains to consider the case when there are exactly two types among $X_1, \dots, X_{n}$. Repeatedly apply (\ref{eq.agm}) by rearranging the sequence so that the first half $S$ contains only the variables of the largest type. Stop when either a single type remains or %after at most $n$ steps
    the first cycle $(S_1, T_1)$, $(S_2, T_2)$, $\dots$, $(S_k, T_k)$ is formed, i.e. the two halves $(S_k, T_k)$ after some step have the same distribution as the two halves $(S_1, T_1)$ in a previous step.
    This procedure is well defined because the number of possible configurations for $S$ and $T$ is finite. Using (\ref{eq.agm}) and additionally ordering the variables by their ``sign'' we can even ensure that the total number of steps is at most $n-1$.  
    Suppose we still have two types in the end. Then
    \[
        \pr(X_1 + \dots + X_{n} = 0) \leq \dots \leq \pr(S_1 +T_1=0) \le \dots \leq \pr(S_k + T_k =0),
\]
    which implies $\pr(S_1 + T_1 = 0) = \dots = \pr(S_k + T_k = 0)$, and so $T_1 \sim -S_1$, see our observation on the equality in (\ref{eq.agm}).

    Taking $j$ such that $X_j$ has the prevailing type completes
    the proof of the stated inequality.
    By the above observation, if we applied (\ref{eq.agm}) at least once where $T \sim -S$ does not hold, this inequality is strict. Otherwise we must have $\sum X_i \sim \sum (-1)^{i+1} Y_i$.
\end{proofof}

\bigskip

\begin{proofof}{Corollary~1}
   Apply Lemma \ref{l1} to the independent random vectors $a_iX_i$.
\end{proofof}

\bigskip

Part of the early inspiration for Lemma~\ref{l1} came from a simple observation of a \emph{math.stackexchange} user André Nicolas about simple symmetric random walks \cite{anicolas}.

\section{Random variables with bounded concentration}

Let $X_1,\ldots, X_n$ be independent random vectors in $\R^d$ and denote their sum by $S_n$. Assume that for all $i$ we have
$$\sup_{x\in \R^d}\pr(X_i=x)\leq \alpha \in (0,1).$$
The bounds on the concentration probability $\pr(S_n=x)$ were studied by many authors. Let us just mention the work of Esseen \cite{ess}, Rogozin \cite{rog} and Gamkrelidze \cite{g}. It was proved by Rogozin that when $d=1$ and $\alpha=\frac{1}{k}$ for $k\in \N$, the probability $\pr(S_n=x)$ is maximized when all $X_i$ are iid uniform random variables in the set $\{0,\ldots, k-1\}$. This result also follows from more general bounds obtained by Leader and Radcliffe \cite{LR}. To our knowledge the sharpest known bounds for $\alpha \in [\frac{1}{2},1)$ and all $d$ were obtained by Ushakov \cite{Ushakov}.
Such $\alpha$ are especially interesting as they cover all Bernoulli distributions. Ushakov established the inequality
\begin{align*}
    \pr(S_n=x) &\leq (2\pi (n+1) \alpha(1-\alpha))^{-\frac{1}{2}}\left(1+(2 (n+1) \alpha(1-\alpha))^{-\frac{1}{2}}\right),
 %  \\ & 
%    =\frac{1}{\sqrt{2\pi n\alpha(1-\alpha)}}(1+O(n^{-\frac 1 2})),
\end{align*}
which is asymptotically sharp: this can be seen by Lemma~9 in \cite{JKarxiv} (with the correct second order term), or alternatively by using the Local Limit Theorem.
%clearly asymptotically sharp, since by the local limit theorem and independent random variables $X_i$ have Bernoulli distributions with parameter $1-\alpha$ we have that $\pr(S_n=n(1-\alpha))$ asymptotically mathes the latter upper bound. 
%We are especially interested in this latter range of 

Let us give a short description of the proof of Theorem \ref{thm2}. Firstly, we characterize the extremal points of the convex set of distributions with a bound on their maximal probability. We then make use of a result of Ushakov \cite{Ushakov} to reduce the problem from high dimensions to integer-valued random variables. Having narrowed down the class of distributions, we use the balancing lemma. The latter step produces a sum of symmetric distributions and we then proceed by using an old rearrangement inequality for convolutions of sequences proved by Gabriel \cite{Gabriel} which we have been fortunate to find in the classical monograph of Hardy, Littlewood and P\'olya \cite{HLP}. After the latter operation the random variables under consideration become symmetric and unimodal. The final touch is to use a discrete analogue of Birnbaum's result from \cite{Birnbaum} on peakedness of symmetric unimodal random variables which intuitively compresses the mass of the underlying distributions to the center as much as it is possible. 

Having outlined the strategy, we shall step by step introduce the relevant notions and results until we can then combine them and finish the proof in a few lines.

For any probability measure $\mu$ on a finite set $X\subset \R^n$ define its concentration
to be the quantity
\[
    Q(\mu) = \max_{x \in X} \mu\{x\}.
\]
Notice that $Q$ is a convex functional. Also note that the set of measures $S_\alpha=\{\mu | Q(\mu) \le \alpha\}$
is convex. Given $\alpha \in (0,1)$, a set $A \subseteq X$ with $|A| = \lfloor \alpha^{-1} \rfloor$ 
and $y \in X \setminus A$, let us denote by $\mu_{\alpha, A,y}$ the probability measure in $S_\alpha$ such that
\[
    \mu_{\alpha, A, y}\{x\} = \begin{cases} 
                 \alpha, &\mbox{for } x \in A, \\
                 1 - \lfloor \alpha^{-1} \rfloor \alpha, &\mbox{for } x = y, \\
                 0, &\mbox{otherwise.}
    \end{cases}
\]
(When $\alpha^{-1}$ is integer it is equal to $|A|$, in this case $y$ becomes a dummy parameter: $\mu_{\alpha, A, y}$ is the uniform measure on $A$ for any $y \in X \setminus A$.)
We shall say that a convex combination $p\mu + (1-p) \nu$ of two distinct measures $\mu$ and $\nu$ on $X$ is \emph{non-trivial} if $0<p<1$.

\begin{lemma}\label{extreme} Let $\alpha \in (0,1)$ and let $\mu$ be a measure in $S_\alpha$. Then $\mu$ can be written as a non-trivial convex combination
    of two distinct measures in $S_\alpha$ if and only if it is not a measure $\mu_{\alpha, A, y}$ for some $A \subseteq X$ and $y \in X\setminus A$.
\end{lemma}
\begin{proof}
    First let us show that if $\mu=\mu_{\alpha, A, y}$, then it cannot be decomposed. Assume the contrapositive:
    that $\mu = p \mu_1 + (1-p) \mu_2$ for distinct measures $\mu_1, \mu_2 \in S_\alpha$ and $0 < p < 1$.
    It follows that both measures have support on $A \cup \{y\}$ if $\alpha^{-1}$ is not integer and on $A$ otherwise.
    In the latter case all measures are equal, a contradiction. In the former case there is $x \in A$ such that $\mu\{x\} = \mu_1\{x\}= \alpha$ and $\mu_2\{x\} = 1 - \lfloor \alpha^{-1} \rfloor \alpha < \alpha$. So $p \mu_1\{x\} + (1-p) \mu_2\{x\} < \alpha=\mu\{x\}$, also a contradiction.

    Now assume that $\mu$ is not of the form $\mu_{\alpha, A, y}$. 
Since $\mu \in S_\alpha$, its support is of size at least $\lfloor \alpha^{-1} \rfloor + 1$.
    Let  $A$ be the set of $\lfloor \alpha^{-1} \rfloor$ largest atoms of $\mu$ and let $y$ be its largest atom outside $A$.
    %in particular it 
    %Note that $\mu$ has a non-zero mass
    Thus $\mu\{x\}>0$ for each $x \in A \cup \{y\}$.
    Let $\mu_2=\mu_{\alpha, A, y}$.
    Fix a positive $\epsilon$ small enough that  %$(1+\epsilon) \mu\{x\}\ge \epsilon ( (1+\epsilon) \mu\{x\} - \epsilon \mu_2\{x\}) \ge 0$
    $(1+\epsilon) \mu\{x\} - \epsilon \mu_2\{x\} \ge 0$
    for $x \in A \cup \{y\}$ and 
    $1+ \epsilon \le \alpha (\lfloor \alpha^{-1} \rfloor + 1)$.
    Define $\mu_1 = (1+\epsilon) \mu - \epsilon \mu_2$. We have $\mu = p \mu_1 + (1-p)\mu_2$ with $p=\frac 1 {1+ \epsilon}$.
    Since $\mu_1 = \mu + \epsilon(\mu - \mu_2)$ and $\mu \ne \mu_2$, 
    %$\mu_1$ and $\mu$ are distinct, and so 
    $\mu_1$ and $\mu_2$ must be distinct.
    Let us now check that $\mu_1 \in S_\alpha$. 
    For $x \in A$ we have $\mu_1\{x\} = \mu\{x\} + \epsilon(\mu\{x\} - \alpha) \le \mu\{x\} \le \alpha$.
    By the choice of $A$ and $y$, for each $x \in X \setminus A$, $\mu\{x\} \le \mu\{y\} \le (|A| + 1)^{-1}$. Thus $\mu_1\{x\} \le (1+\epsilon) \mu\{x\} \le (1+\epsilon) (|A| + 1)^{-1} \le \alpha$ for $x \in X \setminus A$.
   \end{proof}

\medskip

In % his investigation of
his work on %related to
the problem of this section Ushakov \cite{Ushakov} proved a couple of reduction lemmas that allow switching from distributions in Hilbert spaces to distributions on the integers. We shall state the one we require here.

\begin{lemma}\label{reduction}
Let $\mu_{1}, \ldots, \mu_{n}$ be probability distributions in some Hilbert space such that
$$Q(\mu_{i})\leq \alpha.$$
Then there exist probability distributions $\nu_{1}\ldots, \nu_{n}$ on $\Z$ such that $Q(\nu_{i})\leq \alpha$ and
$$Q(\mu_{1} \ast\ldots \ast\mu_{n})\leq Q(\nu_{1}\ast\ldots \ast\nu_{n}),$$ 
where $\ast$ stands for convolution.
\end{lemma}

In two important parts of the proof we shall use rearrangement results from \cite{Gabriel} (see also \cite{HLP} page 273, Theorem 374).  First let us define certain special rearrangements of a finite sequence of non-negative numbers $(a)=(a_{-k},\ldots,a_k)$ indexed by integers. The rearrangement $(^{+}a)$ is defined by inequalities $^{+}a_{0}\geq {^{+}}a_{-1}\geq {^{+}}a_{1}\geq {^{+}}a_{-2}\geq\ldots \geq {^{+}}a_{k}$. Analogously, the rearrangement $(a^{+})$ is defined by inequalities  $a_{0}^{+}\geq a_{1}^{+}\geq a_{-1}^{+}\geq a_{2}^{+}\geq\ldots \geq a_{-k}^{+}$. Finally, if in the sequence $(a)$ all values except the largest one appear an even number of times, we define the symmetric decreasing rearrangement $(a^{\ast})$ by the inequalities  $a^{\ast}_{0}\geq a^{\ast}_{1}= a^{\ast}_{-1}\geq a^{\ast}_{2}=a^{\ast}_{-2}\geq \ldots \geq a^{\ast}_{-k}= a^{\ast}_{k}$. When $a_i = \pr(X=i)$ for a random variable $X$, we will write for brevity $\pr(X=i)^+ = a_i^+$, etc.
\begin{lemma}\label{sym}
Let $(a), (b), (c), (d),\ldots$ be a finite collection of finite sequences of non-negative numbers such that all collections except maybe $(a)$ and $(b)$ have a symmetric decreasing rearrangement. Then
$$\sum_{r+s+t+u+\cdots=0}a_{r}b_{s}c_{t}d_{u}\ldots\leq \sum_{r+s+t+u+\cdots=0} {^{+}a_{r} b^{+}_{s}c^{\ast}_{t}d^{\ast}_{u}}\ldots.$$
\end{lemma}

The final tool we shall require is a discrete counterpart of Birnbaum's \cite{Birnbaum} result on the peakedness of symmetric unimodal distributions. This result might be known, but we could not find it in the literature, so we provide a simple proof for the readers' convenience.
\begin{lemma}\label{bb}
    Let $X$, $Y$ and $Y'$ be independent symmetric unimodal integer random variables.
    Suppose $\pr(Y \in [-k, k]) \le \pr(Y' \in [-k, k])$ for any integer~$k\ge0$. 

    Then for any integer~$k \ge 0$
    \[
        \pr(X + Y \in [-k,k]) \le \pr(X + Y' \in [-k, k])
    \]
\end{lemma}
\begin{proof}
    $|Y'|$ is stochastically
    dominated by $|Y|$, so let us assume $Y'$ and $Y$ are coupled so
    that conditioned on $Y=y$,
    $Y'$ is zero or of the same sign as $y$ and $|Y'| \le |y|$.

    Since $X$ is symmetric and unimodal, for any $y, y' \in \mathbb{Z}$ if $0 \le y' \le y$ or $y \le y' \le 0$
    we have $\pr(X \in [y-k, y+k]) \le \pr(X \in [y'-k, y'+k])$ for any integer $k \ge 0$. 
    Therefore
    \begin{align*}
        & \pr(X + Y \in [-k,k]) = \pr(X - Y \in [-k,k]) = 
        %\\ & 
        \pr(X \in [Y-k,Y+k] ) =
        \\ &
        \E \E(\mathbb{I}_{X \in [Y-k, Y+k]} | Y)
        \le
        %\\ &
        \E \E(\mathbb{I}_{X \in [Y'-k, Y'+k]} | Y) = %= \pr(X - Y' \in [-k, k])
        %\\ & 
        \pr(X + Y' \in [-k, k]).
    \end{align*}
\end{proof}

\medskip

The class of symmetric unimodal distributions is closed under convolution (see e.g. \cite{HLP} Theorem 375). Applying
Lemma~\ref{bb} with the $i$th term of the sum and the rest of the sum for each $i \in \{1,\dots,n\}$ we get:
\begin{corollary}\label{cor3}
    Let $X_{1},\ldots, X_{n},Y_{1},\ldots, Y_{n}$ be independent symmetric unimodal integer random variables. Suppose for $i\in \{1, \dots, n\}$ and any integer $k \ge 0$ we have $\pr(X_i \in [-k, k]) \le \pr(Y_i \in [-k, k])$. Then for any integer $k\ge 0$ we have
$$\pr(X_1+\cdots+X_n \in [-k, k])\leq \pr(Y_1+\cdots+Y_n \in [-k, k]).$$
\end{corollary}

%\textbf{Proof of Theorem \ref{thm2}.} 
\begin{proofof}{Theorem~\ref{thm2}} Let $k = \frac n 2$. Lemma \ref{reduction} tells us that in order to maximize $\pr(X_1+\ldots+X_{2k}=x)$ it is sufficient to consider integer random variables $X_i$ such that $\pr(X_i=x)\leq \alpha$ for all $x\in \Z$. We can without loss of generality also assume that the distribution of each random variable $X_i$ is finitely supported; the general case follows by approximating with truncated random variables. The Krein--Milman theorem \cite{KM} tells us that the convex set of distributions $\mu$ on a finite set $\mathcal{X}\subset \Z$ such that $Q(\mu)\leq \alpha$ is the closure of the convex hull of its extreme points. The extreme points of a convex set $A$ are the points that do not lie in the interior of any segment in $A$. The extreme points for our situation are described by Lemma~\ref{extreme}: they are exactly the measures of the form $\mu_{\alpha, A, y}$ for some $A \subset \mathbb{Z}$ with $|A|=\lfloor \alpha^{-1} \rfloor$ and $y \in \mathbb{Z}\setminus A$. For each $i$ let us define $g(t)=\pr(X_1+\cdots+X_{2k}-X_i+t=x)$. We have $\pr(X_1+\ldots+X_{2k}=x)=\E g(X_i)$ and so if the distribution of $X_i$ is a linear combination of some collection of distributions, then $\E g(X_i)$ is a linear combination of expectations of $g$ with respect to each of these distributions. This means that we can assume that the maximum of $\pr(X_1+\ldots+X_{2k}=x)$ is attained when each $X_i$ has distribution $\mu_{\alpha, A_i, y_i}$ for some $A_i \subset \mathbb{Z}$ and $y_i\in \mathbb{Z} \setminus A_i$. Applying Lemma~\ref{l1} we obtain 
$$\pr(X_1+\ldots+X_{2k}=x)\leq \pr(Y_1-Y_2+\ldots+Y_{2k-1}-Y_{2k}=0),$$
where $Y_{i}$ are iid random variables distributed as some $X_j$. Let us denote the distribution of that particular $X_j$ by $\mu_{\alpha, A, y}$ (dropping the subscripts of $A_i$ and $y_i$). The random variables $Z_i=Y_{2i-1}-Y_{2i}$ are iid and symmetric. %and have a symmetric distribution. 
   % Denote by $Z_{i}^{\ast}$ a sequence of iid random variables that have distributions on the integers obtained by a symmetric decreasing rearrangement of the probabilities $\pr(Z_i = m)$. 
Let $(Z_i^*, i\in\{1, \dots, k\})$
be a sequence of iid random variables where the distribution of $Z_i^*$ is obtained from the symmetric decreasing rearrangement of the distribution of $Z_1$.
    Applying Lemma \ref{sym} we obtain  
\begin{eqnarray*}
\pr(Y_1-Y_2+\ldots+Y_{2k-1}-Y_{2k}=0)&=&\pr(Z_1+\ldots+Z_k=0)\\
&=&\sum_{m_1+\cdots+m_k=0}\pr(Z_1=m_1)\ldots \pr(Z_k=m_k)\\
&\leq& \sum_{m_1+\cdots+m_k=0}\pr(Z_1=m_1)^{\ast}\ldots \pr(Z_k=m_k)^{\ast}\\
&=&\pr(Z_{1}^{\ast}+\ldots+Z_{k}^{\ast}=0).
\end{eqnarray*}
We have now achieved an inequality for the probability in question in terms of symmetric unimodal distributions to which Corollary~\ref{cor3} applies. All that is left to prove is the stochastic domination condition $\pr(|Z_{i}^{\ast}|\leq l)\leq \pr(|U_{2i-1}^\alpha-U_{2i}^\alpha|\leq l)$ for all integers $l$.
    %For this we shall need the rearrangement inequality.
    We shall actually %use it to prove 
    show
    that 
    $$\pr(|Z_i^\ast| \le l) = \max_{B\subset\mathbb{Z}, |B|=2l+1}\pr(Z_i\in B)\leq \pr(|U_{2i-1}^\alpha-U_{2i}^\alpha|\leq l).$$
%If we establish the latter then, by 
    The first equality follows from 
    the definition of the symmetric decreasing rearrangement.
    %, we have that $\max_{B}\pr(Z_{j}\in B)=\pr(|Z_{j}^{\ast}|\leq k)$.
    For the inequality we will use Lemma~\ref{sym} again.

    Let us denote by $U(-B)$ a uniform random variable on the set $\{-x : x\in B\}$ which is independent of the %collection $(Z_j)$.
    previously defined random variables.
    In the case $B=\{-l,\ldots,l\}$ we shall denote this random variable by $U$. Recalling that $Y_1$,$\dots$,$Y_{2k}$ have distribution $\mu_{\alpha, A, y}$ and using Lemma~\ref{sym} we obtain
\begin{eqnarray*}
&&(2l+1)\pr(Z_i\in B)=\pr(Y_{2i-1}-Y_{2i}+U(-B)=0)\\
%%&=&\pr(X_{2i-1}-X_{2i}+U(-B)=0)\\
&=&\sum_{r+s+t=0}\pr(Y_{2i-1}=r)\pr(-Y_{2i}=s)\pr(U(-B)=t)\\
&\le&\sum_{r+s+t=0}{^{+}\pr(Y_{2i-1}=r)\,\pr(-Y_{2i}=s)^{+}\,\pr(U(-B)=t)^{\ast}}\\
    &=&\sum_{r+s+t=0}\mathbb{P}(U^{\alpha}_{2i-1}-\left \lfloor\lfloor \alpha^{-1} \rfloor/2 \right \rfloor=r)\mathbb{P}(-U^{\alpha}_{2i}+\left \lfloor \lfloor \alpha^{-1} \rfloor/2\right \rfloor=s)\mathbb{P}(U=t)\\
&=& \mathbb{P}(U^{\alpha}_{2i-1}-U^{\alpha}_{2i}+ U=0)\\
&=& (2l+1)\mathbb{P}(U^{\alpha}_{2i-1}-U^{\alpha}_{2i}\in \{-l,\ldots,l\})
\end{eqnarray*}
and we are done.
\end{proofof}

\medskip

\begin{proofof}{Corollary \ref{constant}.}
   % For even $n$, apply Theorem~\ref{thm2} and use
   % the Local Central Limit Theorem for the sum of $\frac n 2$
   % iid random variables distributed as $U^\alpha_1 - U^\alpha_2$.
    Let $n$ be even. The Local Limit Theorem, e.g. Theorem~1 of \cite{davis1995elementary}, applied to the sum of $\frac n 2$ iid random variables $U^{\alpha}_{2i} - U^{\alpha}_{2i-1}$ gives
     \begin{align*}
         & \pr(U^{\alpha}_1-U^{\alpha}_2+\cdots+U^{\alpha}_{n-1}-U^{\alpha}_n = 0) = \left( 2 \pi \frac n 2 \Var(U^{\alpha}_1 - U^{\alpha}_2)\right)^{-\frac 1 2} (1+o(1))
        \\ &
         =  (2 \pi n \Var(U^\alpha))^{-\frac 1 2} (1 + o(1)). 
    \end{align*}
    Clearly the conditions of \cite{davis1995elementary}
    are satisfied since the support of $U^\alpha_1 - U^\alpha_2$ is an interval.

    Let $k=\lfloor \alpha^{-1} \rfloor$. 
    By a simple calculation using the fact that $U^\alpha$ is a mixture of $U^{\frac 1 {k+1}}$ with probability $\lambda = (k+1) (1-k\alpha)$ and $U^{\frac 1 k}$ with probability $1-\lambda$ it follows that
    \[
        \Var(U^\alpha) = \frac 1 {12} k (k+1) \alpha (2 + 4k -3 \alpha k^2 - 3 \alpha k).
    \]
    Applying Theorem~\ref{thm2} completes the proof for even $n$.
%    
%    , which we omit, shows that its variance is $c_\alpha^{-2} \pi^{-1}$.
%    Thus by the Central Limit Theorem and the Local Central Limit Theorem
%    \begin{align*}
%    &\mathbb{P}(X_{1}+\cdots+X_{n} = x) \le 
%    %\\& 
%    \mathbb{P}( U_{1}^{\alpha}-U_{2}^{\alpha}+\cdots+U_{2k-1}^{\alpha}-U_{2k}^{\alpha} = 0) \\
%    & =  (2 \pi )^{-1/2} (2^{-1} Var(U^\alpha_1 - U^\alpha_2) n)^{-1/2}
%    %\\& 
%    = c_\alpha n^{-1/2} (1 + o(1)).
%\end{align*}
    For odd $n$ the same asymptotics follow by Remark~\ref{rmk.1}.
\end{proofof}

\section{Bernoulli distributions}

%Recall that by Corollary~1\ref{cor1} for even $n$,
%the optimal solution to Question~\ref{qfox} is to
%take $x=0$, set $\frac n 2$ of the coefficients $a_i$
%to $1$ (or any other constant $c$), and the remaining
%coefficients to $-1$ (or $-c$).
%Lemma~\ref{l1} says that no other assignment can achieve
%this maximum. 
%For odd $n$ an asymptotically optimal upper bound in
%Question~\ref{qfox} follows by Remark~\ref{rmk.1}.
%However, surprisingly, determining which collections of coefficients
%are extremal becomes a much harder and more subtle problem. It turns out,
%that for odd $n$, the answer depends non-trivially on $p$.
%
%%In this section we use
%%a similar technique as in the proof of Lemma~\ref{l1}
%%to give a partial answer to Question~1 for odd $n$ and several related results. 
%%For example, we show that when $p$ is fixed and $n$ is large enough,
%%like for even $n$,
%%the optimal $a_i$ in Question~\ref{qfox} must be $\pm 1$ (or $\pm c$ for a constant $c$).
%%Independently, Singhal~\cite{Singhal} proves a similar result for all $n$.
%%For fixed $p$ Singhal~\cite{Singhal} also determines the asymptotically 
%%optimal proportion of $+1$s and $-1$s using
%%characteristic functions. This proportion depends non-continuously on
%%the number $p$ and can be very far from balanced.

We now focus on %the Littlewood--Offord theory in
the case where  $a_1, \dots, a_n \in \mathbb{R}^d\setminus \{0\}$ and
$X_1, \dots, X_n$ are iid Bernoulli with parameter $p$ ($p = 1 - \alpha$). The dimension $d$
and $p \in (0, \frac 1 2]$ will be fixed.

 Define
\begin{equation}\label{eq.Tn}
    T_n = T_{n,p} = \bern_1 - \bern_2 + \dots + (-1)^{n+1} \bern_n.
\end{equation}

Note that $T_n \sim Binom(\lceil \frac n 2 \rceil, p) - Binom(\lfloor \frac n 2 \rfloor,p)$, where
by a difference of distributions we denote the distribution of the difference of independent
random variables from the corresponding distributions.

By Corollary~1\ref{cor1} for even $n$ we have
\[
    \pr(\sum a_i X_i = x) \le \pr(T_n = 0).
\]
%gives the best possible inequality in (\ref{LO}) when $n$ is even.
The situation for odd $n$ is more subtle. We are still able to prove the following. 
\begin{lemma}\label{lem.odd_bernoulli}
    Let $p \in (0, \frac 1 2]$.
    For all $n$ large enough the following holds.
    If $\bern_1, \bern_2, \dots, \bern_n$ are independent Bernoulli random variables with parameter $p$, $a_1, \dots, a_n \in \mathbb{R}^d\setminus\{0\}$ and $x \in \mathbb{R}^d$ then
\begin{equation}\label{eq.odd}
    \pr(\sum a_i \bern_i = x) \le \max_{0\le k \le n} \max_{x' \in \{\lfloor \mu \rfloor, \lceil \mu \rceil\}} \pr(B_{n-k,p} - B_{k,p}' = x'),
\end{equation}
    where $B_{n-k,p}, B_{k,p}'$ are independent binomial random variables with parameters $(n-k, p)$ and $(k, p)$ respectively, and $\mu = (n-2k)p = \E B_{n-k,p} - \E B_{k,p}'$.
\end{lemma}

By Lemma~\ref{l1} for even $n$ the only value that maximizes the right side of (\ref{eq.odd}) is $k=\frac n 2$.
%, and this holds for any $n$.
For odd $n$ we cannot explicitly describe the optimal $k$ in (\ref{eq.odd}). 
As $\max_x \pr(B_{n-k,p} - B_{k,p}' = x)$ has the same order asymptotic growth for any $k=k(n)$, the answer requires maximizing the second order asymptotic term over all $k$. It seems that for a given $n$ each $k \in \{0, \dots, \frac {n-1} 2\}$ can
be optimal depending on $p$ in a complicated way (we confirmed this using \texttt{distr} package of \texttt{R} \cite{distr} for, e.g., $n \le 31$). One can use Lemma~\ref{lem.large_odd} and Lemma~\ref{lem.taylorTn} stated below in this section to show that for sufficiently large odd $n$ neither of $k \in \{0, \frac{n-3} 2, \frac{n-1} 2 \}$ dominates for all $p \in (0,0.5)$; this provides a counterexample to a conjecture in an early version of~\cite{FKS}.

Independently, Singhal~\cite{Singhal} proves a result similar to Lemma~\ref{lem.odd_bernoulli} for all $n$.
For fixed $p$ he also determines the asymptotically 
optimal proportion of $+1$s and $-1$s (i.e. $k/n$) using
characteristic functions. This proportion depends non-continuously on
the number $p$ and can be very far from balanced.

Lemma~\ref{lem.odd_bernoulli} follows from a slightly stronger result, Lemma~\ref{lem.large_odd}, which we prove next.
Recall that random vectors $X$ and $Y$ have the same type if and only if either $Y$ or $-Y$ has the same
distribution as $X$. Similarly we say that $x,y \in \mathbb{R}^d$ have the same type if and only if $x \in \{-y, y\}$.

\begin{lemma}\label{lem.large_odd}
    Let $p \in (0, \frac 1 2]$. Let $\bern_1, \dots, \bern_n$
    be independent Bernoulli random variables with parameter $p$. 
    Let $T_n$ be as in (\ref{eq.Tn}).

    There is a sequence $\delta_n = o(n^{-1})$ such that
    for any\footnote{Here the $d$-dimensional zero vector is also denoted as 0.} $a_1, \dots, a_n \in \mathbb{R}^d \setminus \{0\}$ with 
    at least two types and any $x \in \mathbb{R}^d$ then either (a) all but one of the coefficients are of the same type or (b)
    \[
        \pr(\sum_{i=1}^n a_i \bern_i = x) \le \pr(T_n = 0) (1 - (2n)^{-1} + \delta_n).
    \]
\end{lemma}

We will need asymptotics for small deviations of a binomial random variable convolved with its negation.

\begin{lemma} \label{lem.smalldev}
    Let $p \in (0, \frac 1 2]$. Let $k$ be a positive integer.
    Let $T_{2n} = T_{2n,p} \sim Binom(n,p) - Binom(n,p)$.
    Then
    \begin{align*}
%        &\frac {\pr(S_{2n} = 2k)} {\pr(S_{2n} = 0)} = 1 - \frac {2 k^2} {p n} (1 + o(1)); \\
%        &\frac {\pr(S_{2n} = 2k+1)} {\pr(S_{2n} = 1)} = 1 - \frac {2 k (k+1)} {p n} (1 + o(1)); \\
        &\frac {\pr(T_{2n} = k)} {\pr(T_{2n} = 0)} = 1 - \frac {k^2} {4 p (1-p) n} (1 + o(1)).
    \end{align*}
\end{lemma}
\begin{proof}
    We have 
    \[
        T_{2n} = \bern_1 + \dots  + \bern_n -  \bern_1' - \dots -  \bern_n', 
    \]
    where $\bern_j, \bern_j'$, $j \in \{1, \dots,n\}$ are independent Bernoulli with parameter $p$.

    On the other hand $\bern_j - \bern_j'$ is distributed as a mixture of 0 (with probability $1-p' = (1-2p)^2$) and $\tilde{\bern}_{2j} + \tilde{\bern}_{2j+1} - 1$ (with probability $p' = 4 p (1-p)$). Here $\tilde{\bern}_{2j}$ and $\tilde{\bern}_{2j+1}$ are two independent Bernoulli random variables with parameter $\frac 1 2$.
 
    Thus
    \[
        T_{2n} = -N + B_{2N} \mbox{ where } N \sim Binom(n, p')
    \]
    and conditioned on $N=t$, $B_{2N} \sim Binom(2t, \frac 1 2)$.

    Conditioned on $N = t$, $t\ne 0$, $T_{2n}$ is unimodal and symmetric with mode at 0.
    For any positive integer $k$
    \begin{align*}
        &\frac {\pr(T_{2n}=k | N = t)} {\pr(T_{2n}=0 | N = t)} = \frac {\binom {2t} {t+k}} {\binom {2t} t} 
        %\\ & 
        = \frac{ (t-k+1) \dots t} {(t+1) \dots (t+k)} 
        %\\ &
        =\frac {\prod_{j=0}^{k-1} (1-j/t) } {\prod_{j=1}^k (1 + j/t)} \\
        &=(1 - \sum_{j=1}^{k-1} \frac j t + O(t^{-2})) (1 - \sum_{j=1}^k \frac j t + O(t^{-2})) = 1-\frac {k^2}  t + O(t^{-2}).
    \end{align*}
    If $p = \frac 1 2$ we have $\pr(N=n)=1$ and the lemma follows, so we will assume $p < \frac 1 2$.
    Now let $A_C$ be the event
    $|N - np'| \le C \sqrt{n \ln n}$.
    Fix $C$ large enough so that by the concentration of the binomial random variable $N$ we have $\pr(\bar{A}_C) = o(n^{-3/2})$ (by, e.g., Theorem~2.1 of \cite{cmcd98}, $C$ can be any constant larger than $\sqrt{3}/2$).

    For any $t$ such that $A_C$ holds on $N=t$ %we have uniform bound as $n\to \infty$
    \begin{align*}
        &\pr(T_{2n}=k | N=t) = \pr(T_{2n}=0 | N=t) (1 - \frac {k^2} {np'} (1+o(1)))
    \end{align*}
    where the constant in $o()$ depends only on $k$ and $C$.
    By
    Stirling's approximation, see also Lemma~\ref{lem.taylorTn} below, %\m{add?}
    %(\ref{eq.stirling})
    we have $\pr(T_{2n}=0) = \Theta(n^{-\frac 1 2})$. %2^{1/2}(\pi np')^{-1/2} (1+o(1))$.
    So
    \begin{align*}
        &\pr(T_{2n} = k) = \E\E (\mathbb{I}_{A_C} \mathbb{I}_{T_{2n}=k} | N)  +   \E\E (\mathbb{I}_{\bar{A}_C} \mathbb{I}_{T_{2n}=k} | N); 
        \\ &
        \E\E (\mathbb{I}_{\bar{A}_C} \mathbb{I}_{T_{2n}=k} | N) \le \pr(\bar{A}_C) = o(n^{-1} \pr(T_{2n}=0)); 
        \\ &
        \E\E (\mathbb{I}_{A_C} \mathbb{I}_{T_{2n}=k} | N) =  \E\E (\mathbb{I}_{A_C} \mathbb{I}_{T_{2n}=0} | N) \left(1-\frac{k^2} {np'} + o(n^{-1})\right)
    \end{align*}
%        \E \E (\bern_{A_C} \bern_{S_{2n}=k} | N) \le \pr(S_{2n}=k) \le \E \E (\bern_{A_C} \bern_{S_{2n}=k} | N) + \pr(\bar{A}_C),
%    \intertext{which gives}
%        &\pr(T_{2n}=k) = \pr(T_{2n}=0) \left(1 - \frac {k^2} {np'} (1+o(1))\right). % + o(n^{-3/2}).
    which completes the proof since
    \begin{align*}
        \pr(T_{2n} = 0)  \ge  \E\E (\mathbb{I}_{A_C} \mathbb{I}_{T_{2n}=0} | N) \ge  \pr(T_{2n} = 0) - \pr(\bar{A}_C) = \pr(T_{2n} = 0) (1-o(n^{-1})). 
    \end{align*}
\end{proof}

\bigskip

\begin{proofof}{Lemma~\ref{lem.large_odd}}
    Let $S_n = \sum_{i=1}^n a_i \bern_i$. Then 
    \[
        \pr(S_n = x) = (1-p) \pr(S_{n-1} = x) + p \pr(S_{n-1} = x-a_n) \le \pr(S_{n-1}=x')
    \]
    for some $x' \in \{x, x-a_n\}$.
    Let $m = \lfloor \frac n 2 \rfloor$ and consider $S_{2m}$. 
    Define $n_a = |\{i: i \le 2m \mbox{ and }  a_i \in \{-a,a\}\}|$ and let $c$ maximize $n_a$ over $a$.

    Suppose first $n_c \le 2m-2$. % We can assume $x_{2m} \ne \pm c$.

    Let $x' \in \mathbb{R}^d$ be arbitrary.
    By (\ref{eq.agm}) applied to random vectors
    $a_1 \bern_1$, $\dots$, $a_{2m-1} \bern_{2m-1}$, $a_{2m} \bern_{2m} - x'$,
    %, $i\in\{1,\dots, 2m-1\}$
    %and $X_{2m}=x_{2m} \bern_{2m} - x'$
    placing exactly $\lceil \frac {n_c} 2 \rceil$ of the terms
    with $a_i \in \{-c,c\}$ into the first half, we get
    \[
        \pr(S_{2m} = x') \le \pr(\sum_{i=1}^{2m} {a_i'} \bern_i = 0)
    \]
    %where the multiplicity of each element in $\{a_i'\}$ is equal to the multiplicity of its negation. 
    Denote by $n_a$ the sum of multiplicities of $a$ and $-a$ in $\{a_i': i \in \{1, \dots, 2m\}\}$. As these multiplicities are equal, see the proof of Lemma~\ref{l1}, $n_a'$ is even, and by our choice of ordering $n_a' \le 2m-2$. So there are at least two equivalence classes (types) of random variables in the resulting sum.
    
    We claim that we can keep applying (\ref{eq.agm}) as in the proof of Lemma~\ref{l1},
    and stop when exactly two types of random variables remain.
    
    %We can ensure that after each application we have at least two types of random variables,
    %and one less type as follows.
    %by placing
    %exactly half of the variables from the largest equivalence class into the first and the second half respectively,
    %and placing the terms in the smallest and the second smallest equivalence classes into
    %different halves until at most three equivalence classes remain.
    If there are at least four equivalence classes, place half of the $n_1$ variables from the largest equivalence class into the first half of the sum, and the remaining $\frac{n_1}2$ variables into the second half of the sum. Note that we can then place the variables from the smallest and the second smallest equivalence classes into different halves, so that after an application of (\ref{eq.agm}) the number of classes is reduced by at least one and remains at least two. 
    %vk 2019-08 pamiršau koks tikslas išskaidyti n_1 į dvi lygias dalis. Kad išlaikyti didžiausią klasę, ne, atrodo, tam, kad nepasidarytų 1.

    If there are three equivalence classes, arrange the sum so that the $n_2$ terms
    from the second biggest class go first, then the $n_1$ terms from the biggest class
    and finally the $n_3$ terms from the smallest class. We have $n_1 + n_2 \ge 2 (2 m)/3$.
    As $n_1, n_2, n_3 \ge 2$, $n_2 \le (2m - n_3)/2 \le m-1$. Hence the first 
    half contains all $n_2$ elements from the second biggest class, and at least one element from
    the biggest class, while the second half contains $n_1 + n_2 - m \ge \frac m 3$ elements
    of the biggest class and all $n_3$ elements from the smallest class. Applying (\ref{eq.agm})
    one more time, exactly two equivalence classes remain. Hence after at most $n$ applications of (\ref{eq.agm})
    in total we get
    \begin{equation}\label{eq.2binomial}
        \pr(S_{2m} = x') \le \pi(a,b,k) := \pr(a T_{2k} + b T_{2(m-k)}' =0)
    \end{equation}
    where $a, b \in \mathbb{R}^d \setminus \{0\}$, $a \not \in \{-b, b\}$,  $k, m-k \ge 1$ and
    $T_{2k}, T_{2(m-k)}'$ are independent, $T_{2(m-k)}' \sim T_{2(m-k)}$.
    %, and $T_{2j}$ and $T_{2j}'$ is distributed as the difference of two independent binomial random variables $Binom(j, p)$.
    %, $j \in\{k, m-k\}$
    % and all the binomial random variables are independent.

    Denote by $S_{2m}'$ the sum of $2m$ Bernoulli($p$) random variables with multipliers
    $a,-a, b$ and $-b$ corresponding to the right side of (\ref{eq.2binomial}). 
    We claim that 
    %\m{might also work for any $j \le min(k, n-k)$: is this interesting?}
    \begin{equation*}
        \pi(a,b,k) \le \max (\pi(a,b,1), \pi(a,b,m-1)).
    \end{equation*}
    To see this, keep applying (\ref{eq.agm}) as in the proof of Lemma~\ref{l1}
    starting with $S_{2m}'$ by taking exactly one term in the smaller equivalence class
    in the first half $S$. Either at some point we get that $\pi(a,b,k) \le \pr(S + T =0)$ with $T \sim - S$,
    or we obtain a cycle in a finite number of applications, where again as in the proof of Lemma~\ref{l1} the same inequality with $S \sim -T$ must hold. %we get that $T \sim -S$ so that $\pi(a,b,k) \le \pr(S-S'=0)$.

    Thus by possibly swapping $a$ and $b$
    \[
        \pr(S_{2m} = x') \le \pr(a T_{2(m-1)} + b T_2' = 0).
    \]
    %Let %$X=B_{m-1} - B_{m-1}'$ and
    %$Y= B_1 - B_1'$.
    Let $X = T_{2(m-1)}$ and $Y = T_2'$.
    Note that whenever $b \ne ra$ for some $r \in \mathbb{Z}$, $a X + b Y = 0$ if and only if $X=0$ and $Y=0$,
    whereas if $b=2a$ and $m\ge 2$ there are additional possibilities for $aX + bY=0$, so we can assume $b = r a$
    for $r \in \mathbb{Z} \setminus \{-1, 0, 1\}$, which reduces the right part of the last inequality to the one-dimensional case $a=1$ and $b=r$.

    Furthermore, $X$ is symmetric and (strongly) unimodal with mode 0 \cite{darroch}.
    Therefore, if $r \in \mathbb{Z}$ and $|r| > 2$, we have
    \begin{align*}
        &\pr(X + r Y = 0) = \sum_k \pr(X=-r k) \pr(Y = k) 
        \\ &
        \le \sum_k \pr(X=-2k) \pr(Y=k) = \pr(X + 2Y = 0).
    \end{align*}
    We have proved that
    \begin{equation}\label{eq.inter1}
        \pr(S_n = x) \le \max_{x'} \pr(S_{2m}=x') \le \pr(X+2Y=0).
    \end{equation}
    %What remains is a simple but a bit tedious comparison of candidates where all but 0, 1 or 2 weights
    %are the same.
    By Lemma~\ref{lem.smalldev} 
    \begin{align}\label{eq.smalldev}
        & \pr(X=k) = \pr(X=-k) = \pr(X=0) \left(1 - \frac {k^2} {2p(1-p)n} + o(n^{-1})\right).
        %\\&  \pr(X=3) = \pr(X=1) \left(1 - \frac 4 {pn} + o(n^{-1})\right) 
    \end{align}
%    As $X$ is strongly unimodal \cite{keilsongerber}, it is log-concave, so
%    \begin{equation}\label{eq.logconc}
%        \pr(X=1) \ge \sqrt{\pr(X=0) \pr(X=2)} \ge \pr(X=0) (1 - \frac 1 {pn} + o(n^{-1})).
%    \end{equation}

    For odd $n$ let us now compare $\pr(X+2Y=0)$ with $\pr(X + Y'=0)$ where $Y' = \bern_1' - \bern_2' + \bern_3'$
    and $\bern_1', \bern_2', \bern_3'$ are independent Bernoulli($p$) random variables independent of $X$. Note that $X+Y' \sim T_n$.

    Note that $2 Y$ is symmetric and distributed on $\{-2, 0, 2\}$ with $\pr(2 Y = 0) = p^2 + (1-p)^2$. Therefore by (\ref{eq.smalldev})
    \begin{align*}
        & \frac {\pr(X + 2Y = 0)} {\pr(X=0)} = 2 p(1-p) \left(1 - \frac 4 {2p(1-p)n}\right) + o(n^{-1}) + p^2 + (1-p)^2
        \\ & = 1 - \frac 4 n + o(n^{-1}).
    \end{align*}
    We have that $Y'$ is distributed on $\{-1, 0, 1, 2\}$ with probabilities $p (1-p)^2$, $2p^2(1-p) + (1-p)^3$, $2p(1-p)^2 + p^3$ and $p^2(1-p)$
    respectively. Therefore by (\ref{eq.smalldev}) %, (\ref{eq.logconc}) 
    and symmetry of $X$
    \begin{align*}
        &\pr(X + Y' = 0) = \sum_{j=-1}^{2} \pr(X=-j) \pr(Y'=j) %\pr(X=1)\pr(Y'=-1) + \pr(X=0)\pr(Y'=0) + 
       % \\ & + \pr(X=-1)\pr(Y'=1) + \pr(X=-2)\pr(Y'=2) =
        \\ &
        = \left(  (3 p (1-p)^2 + p^3) \left(1 - \frac 1 {2p(1-p)n}\right) + 2p^2 (1-p) + (1-p)^3 + \right .
        \\& 
        \left . \quad \quad + p^2 (1-p) \left(1-\frac 4 {2p(1-p)n}\right) + o(n^{-1}) \right) \pr(X=0)
        \\ &
        %= \left(1 - \frac {3(1-p) + p^2/(1-p)  + 4 p} {2n} + o(n^{-1}) \right) \pr(X=0)
        %= \left(1 - \frac {p^2 + 2p(1-p) + 3(1-p)^2} {n} + o(n^{-1}) \right) \pr(X=0)
        \\ &
        %= \left(1 - \frac {2 p^2 - 4p + 3} {n} + o(n^{-1}) \right) \pr(X=0).
        %= \left(1 - \frac {3 + p + p^2/(1-p)} {2n} + o(n^{-1}) \right) \pr(X=0).
        = \left(1 - \frac {3 - 2p} {2 n (1-p)} + o(n^{-1}) \right) \pr(X=0).
    \end{align*}
    Finally, combining (\ref{eq.inter1}) and the last two bounds
    \begin{align*}
        %&\pr(S_{2n} = 0) \le (1 - \frac {4 (1-p)} n + o(n^{-1})) \pr(X=0) \le  (1 - \frac {4 (1-p)} n + o(n^{-1})) \times
        \pr(S_n = x) &\le (1 - \frac 4 n + o(n^{-1})) \pr(X=0) = (1 - \frac 4 n + o(n^{-1})) \times
        \\ &
        %\times \pr(T_{2n} = 0)/\left(1 - \frac {2 p^2 - 4p + 3} {n} + o(n^{-1}) \right)^{-1} 
        %\times \pr(T_{n} = 0) \left(1 - \frac {3 + p + p^2/(1-p)} {2n} + o(n^{-1}) \right)^{-1} 
        \quad \times \pr(T_{n} = 0) \left(1 - \frac {3 - 2p} {2n (1-p)} + o(n^{-1}) \right)^{-1} 
        %\\ &
        %\le \pr(T_{n} = 0) (1 +  \frac {3 + p + p^2/(1-p) - 8} {2n} + o(n^{-1}))
        \\&
        = \pr(T_{n} = 0) \left(1 - \frac {5 - 6p} {2n (1-p)} + o(n^{-1})\right).
    \end{align*}
    %Since $p < \frac 1 2$,  $5 - 6p \ge 2 > 0$, this
    %completes the proof for the case $n$ is odd and $n_c \le 2m -2$.
    If $n_c =2m - 1$ and $a_n$ is not in the largest class (type), we can exchange
    $a_n$ with a constant from the largest type and apply the proof for $n_c = 2m-2$.
    Otherwise all but one $\{a_i\}$ are of the same type, and the result follows.
%    \[
%        \pr(S_n = x) \le \max\{\pr(T_{2m} + 2\bern_1' = 0), \pr(T_{2m} + 2 \bern_1' = 1) \}.
%    \]
%    By Lemma~\ref{lem.smalldev} and similar calculations as above
%    \begin{align*}
%        & \pr(T_{2m} + 2 \bern_1' = 0) = \pr(T_n = 0) \left(1 - \frac 3 {2(1-p)n}+ o(n^{-1})\right); \\
%        %&  \pr(T_{2m} + 2 \bern_1' = 1) = \pr(T_n = 0) \left(1 - \frac 1 {2p(1-p) n} + \frac 1 n + o(n^{-1})\right).
%        &  \pr(T_{2m} + 2 \bern_1' = 1) = \pr(T_n = 0) \left(1 - \frac 1 {2p n} + o(n^{-1})\right).
%    \end{align*}

    Now assume $n$ is even. We only need to consider the case $n_c < n-1$. Using (\ref{eq.smalldev}) similarly as above
    \begin{align*}
        &\pr(S_n = x) \le \pr(T_{n-2} + 2 (X_1' - X_2') =0)
        \\ &
        =\pr(T_{n-2}=0) \left(1 - \frac {2 \cdot 4 p (1-p)} {2 p (1-p) n} + o(n^{-1})\right) = \pr(T_{n-2} = 0) (1 - 4n^{-1} + o(n^{-1}));
          \\ &
        \pr(T_n = 0) = %\pr(T_{n-2} = 0) (1 - \frac {2 p (1-p)} {2 p(1-p) n} + o(n^{-1})) = 
        %\\ &
        \pr(T_{n-2} = 0) (1-n^{-1} + o(n^{-1}))
    \end{align*}
    which implies $\pr(S_n = x) \le (1 - 3n^{-1} + o(n^{-1}))$.
    Combining the bounds for even and odd $n$ if
    %for odd $n$ $\pr(S_n=x) \le \pr(T_n = 0) (1 - n^{-1} %\min( \frac 3 {2(1-p)}, 
    %\frac {5-6p} {2 (1-p)}
    %, \frac 1 {2p})
    %+ o(n^{-1})) % \le  \pr(T_n = 0) (1 - \frac 1 n + o(n^{-1}))
    %$ %and for even $n$ 
    %$\pr(S_n = x) \le \pr(T_n=0) (1 - \frac 3 {2n} + o(n^{-1}))$.
    %$
    \begin{align*}
        \pr(S_n = x) &\le \pr(T_n=0) \left(1 - n^{-1}\min\left(\frac {5-6p} {2 (1-p)}, 3\right) + o(n^{-1})\right) 
        \\ & 
         \le \pr(T_n=0) \left(1 - (2n)^{-1} + o(n^{-1})\right)
    \end{align*}
    as required.
    %$.
%Patikrinau 2019-09-21. Pasiekiama su p~0 ir T_{n-1} + 2I_1 lyginiam n.
\end{proofof}

\medskip

\begin{proofof}{Lemma~\ref{lem.odd_bernoulli}}
    By Lemma~\ref{lem.large_odd} for $n$ large enough $\max_{(a_i), x} \pr(\sum a_i \bern_i=x)$ can only be achieved when all but at most one cofficient $a_i$ satisfy $a_i \in \{-c,c\}$ for some $c \in \mathbb{R}^d$. Without loss of generality suppose $a_i \in \{-c, c\}$ for $i \le n-1$. 
    We have $S_n = \sum_{i=1}^n a_i X_i = c \tilde{S}_{n-1} + a_n X_n$ where  $\tilde{S}_{n-1}=\sum_{i=1}^{n-1} \tilde{a}_i X_i$ for some $\tilde{a}_1, \dots, \tilde{a}_{n-1} \in  \{-1,1\}$.

    Darroch \cite{darroch} proved that if a finite sum of independent Bernoulli random variables has mean $\mu'$
    then it is unimodal and has mode (largest atom) equal to $\lfloor \mu' \rfloor$ or $\lceil \mu' \rceil$, or both.
    Furthermore, the distribution is strictly increasing up to its mode and strictly decreasing past its mode. 
    Applying the results of \cite{darroch} to $\tilde{S}_{n-1}$ it follows that $\pr(S_n=x)$ can only be maximized with $x = c m$ where $m$ is a mode of $\tilde{S}_{n-1}$ and $a_n \in \{-c, c\}$ (while the optimal choice restricted to $a_n \not \in \{-c,c\}$ is $a_n \in \{-2c, 2c\}$). 

    Thus, without loss of generality, the maximum probability is achieved with $d=1$, $a_1, \dots, a_n \in \{-1,1\}$ and for each $k$ it suffices to consider only two values of $x$ by \cite{darroch}. 
\end{proofof}

\bigskip

%To prove Proposition~\ref{rmk.counterexample} we need precise asymptotics for $\pr(T_n = 0)$.
%using a result of Wagner \cite{wagner}.
%A similar proof can be used to compute arbitrarily many terms in the asymptotic expansion.
%Recall that $\pr(T_{2n} = 0)$ is the probablity that two
%independent $Binom(n,p)$ random variables are equal.
Finally, we show how to get precise asymptotics for $\pr(T_n = 0)$, the method can be extended to any number of
lower order asymptotic terms.

\begin{lemma}\label{lem.taylorTn}
    Let $p \in (0, \frac 1 2)$ and $T_n$ be as in (\ref{eq.Tn}). 
    $\pr(T_n=x)$ is maximized at the unique point $x=0$.   
    \vspace{-16pt}
    \begin{align*}
    \intertext{\quad For even $n$}
        \pr(T_{n} = 0) &= \sum_{k=0}^n \pr(B_{n,p}=k)^2
        \\ &
        = \frac 1 {\sqrt{2\pi n p (1-p)}} \left(1 + \frac 1 {4n} \left(\frac 1 {2p(1-p)} - 3\right) + O(n^{-2}) \right). 
    %\end{align*}
    \intertext{\quad For odd $n$}
    %\begin{align*}
        \pr(T_{n} = 0) &= \frac 1 {\sqrt{2\pi n p (1-p)}} \left(1 + \frac {2p^2 - 6p + 1} {8 n p (1-p)} + o(n^{-1}) \right). 
    \end{align*}
\end{lemma}

\begin{proof}
    For even $n$ the maximum atom of $T_n$ is 0 by symmetry and unimodality of $T_n$ \cite{darroch}.
    For odd $n$ since $T_{n-1}$ is unimodal, it is easy to see that $T_{n}$ is also unimodal and 
    \begin{equation}\label{eq.alternation}
        \pr(T_{n} = 0) > \pr(T_{n}=1) > \pr(T_{n}=-1) > \pr(T_{n}=2) > \dots,
    \end{equation}
    in particular, for odd $n$ the maximum atom is still zero. % by (\ref{eq.alternation}).

    Wagner \cite{wagner} presented a simple argument using analytic combinatorics and the generating function for the values of Legendre polynomials, that for $b > 0$ and $c>0$
    \begin{equation}\label{eq.wagner}
        [x^n] (x^2+bx+c)^n = \frac {(b+2\sqrt{c})^{n+1/2}} {2 c^{1/4} \sqrt{\pi n}} \left(1+\frac {b-4 \sqrt c} {16 n \sqrt c} + O(n^{-2})\right).
    \end{equation}
    We have 
    \begin{align*}
        \pr(T_{2n} = 0) &= (1-p)^{2n} \sum \binom n k^2 \frac {p^{2k}} {(1-p)^{2k}}
        %\\ &
        = (1-p)^{2n} [x^n] (1 + d x)^n (x+1)^n 
        \\ &
        = %(1-p)^{2n} d^n
        p^{2n} [x^n](x^2 + \frac {d+1} d x + d^{-1})^{n}
    \end{align*}
    where $d = \frac {p^2} {(1-p)^2}$.
    Let
    %\m{hide some of simple calcs in this proof}
    \begin{align*}
        & a=1, \quad b=\frac{d+1} d = \frac {2p^2 - 2p + 1} {p^2} \quad\mbox{and}\quad c=d^{-1} = \frac {(1-p)^2} {p^2}.
       % \\ &
    \end{align*}
    Then
    \begin{align*}
     %   \sqrt{c} = \frac {1-p} p; %\quad b + 2\sqrt c = p^{-2};
     %   \quad  c^{1/4} = (1-p)^{1/2} p^{-1/2};
     %   \\
        & b + 2 \sqrt c = \frac {2p^2 - 2p + 1} {p^2} + \frac {2(1-p)} p = p^{-2}
        \\ & \frac {b - 4 \sqrt c} {2\sqrt c} = \frac {b+2\sqrt c} {2 \sqrt c} - 3 = \frac 1 {2 p (1-p)} - 3;
        \\ & \frac {(b+2 \sqrt c)^{n+1/2}} {2 c^{1/4}} = \frac {p^{-2n - 1}} {2 p^{-1/2} (1-p)^{1/2}} = 2^{-1} p^{-2n} p^{-1/2} (1-p)^{-1/2}
    \end{align*}    
    and by  (\ref{eq.wagner})
    \[
        \pr(T_{2n} = 0) = \frac  1 {2 \sqrt{\pi np (1-p)}} \left(1 + \frac 1 {8n} \left(\frac 1 {2p(1-p)} - 3\right) + O(n^{-2}) \right). 
    \] 
    For odd $n$ by Lemma~\ref{lem.smalldev} 
    \begin{align*}
        \pr(T_n = 0) &= p \pr(T_{n-1} = -1) + (1-p) \pr(T_{n-1} = 0)
        \\ &
        =\pr(T_{n-1} = 0) \left(1 - \frac {p}{2 p (1-p) n}  + o(n^{-1})\right)
        \\&
        =\pr(T_{n-1}=0) \left(1 - \frac 1 { 2(1-p)n}  + o(n^{-1})\right) .
    \end{align*}
    Using the already proved part for even $n$
    \begin{align*}
        \frac {\pr(T_{n-1} = 0)} {\sqrt{2\pi n p (1-p)}} = \left( 1 + \frac 1 {4n} \left(\frac 1 {2p(1-p)} - 3\right) + O(n^{-2}) \right) \left(1 + \frac 1 {2n} + O(n^{-2})\right).
    \end{align*}
    Combining the last two estimates 
    \begin{align*}
        & \frac {\pr(T_n = 0)} {\sqrt{2\pi n p (1-p)}  } =
        1 + n^{-1}\left(\frac 1 {8p(1-p)} - \frac 3 4 - \frac 1 {2(1-p)} + \frac 1 2 \right) + o(n^{-1})
        \\ &
        = 1 + \frac {2p^2 - 6p + 1} {8 n p (1-p)} + o(n^{-1}).
    \end{align*}
\end{proof}

\section{Open problems and concluding remarks}

We believe that at least for lattice-valued random vectors the following more general result is true.
\begin{conjecture}\label{conj.1}
Let $X_1,\ldots, X_n$ be iid random vectors in $\Z^d$. Then there exists a choice of weights $w_i\in \{-1,1\}$ such that for all non-zero $a_i\in \R$ and all $x\in \R^d$ we have
$$\pr(a_1X_1+\ldots+a_nX_n = x)\leq \max_{k\in \Z^d}\pr(w_1X_1+\ldots+w_nX_n=k).$$
\end{conjecture}
Of course, in view of Corollary~1\ref{cor2} one would have to only prove it for odd $n$.

The second conjecture concerns Theorem \ref{thm2}.
\begin{conjecture}
Let $X_1,\ldots, X_n$ be iid random vectors in $\R^d$ such that 
$$\sup_{x\in \R^d}\pr(X_i=x)\leq \alpha.$$
Then there exists a choice of weights $w_i\in \{-1,1\}$ such that for all non-zero $a_i\in \R$ and all $x\in \R^d$ we have
$$\pr(X_1+\ldots+X_n = x)\leq \max_{k\in \Z}\pr(w_1U^{\alpha}_1+\ldots+w_nU^{\alpha}_n=k).$$
\end{conjecture}

For the simplest case of iid Bernoulli random variables $X_i$ (i.e. $\alpha \ge \frac 1 2$)
we saw that this is true for even $n$ and large odd $n$, and Singhal \cite{Singhal} independently proved it for all~$n$.
What is the optimal number $l=l(n,p)$ of +1s? %The answer for $p$ fixed and $n$ large enough (depending on $p$) was recently obtained by Singhal. 
Although Singhal obtains excellent results for fixed $p$ and large $n$, 
the complete answer still seems not obvious. For example, for a given $n$, can $l$ take any value in $\{0, \dots, n\}$ depending on $p$? 

%But the optimal number $k^*$ of $+1$s as a function of $p$ and $n$ is far from obvious. Can $k^*$ be any value in $\{0, \dots, n\}$ depending on $p$? Can an explicit formula be obtained?

Tao and Vu proved in \cite{tv} that for a collection of non-zero $a_i\in \Z^d$ and independent random variables $X_i$ such that $\pr(X_i=\pm 1)=\frac{1}{2}$ if the probability $\pr(a_1X_1+\cdots+a_nX_n=x)$ is large, then most of the $a_i$ can be covered
by a small number of generalized arithmetic progressions. In other words, the collection of weights $a_i$ has strong additive structure. 
%This seminal work and its refinements have become an important in the investigation of random matrices. We believe that this principle is universal
%and not so much dependent on the distribution of the random variables $X_i$.
Their work lead important progress in the investigation of random matrices.
\begin{question}\label{q.2}
Can inverse statements of Corollary~1\ref{cor1} be obtained if we additionally assume that $a_i\in \Z^d$?
\end{question}
%The latter formulation is rather vague. 
In the case when the variances of $X_i$ are bounded, we believe that the inverse statements should be 
%verbatum 
analogous to the corresponding ones in \cite{tv}. The precise statement
might need to be formulated differently in the case when $X_i$s have
a heavy-tailed distribution.
%We are not sure how the precise statement should be formulated in the case when $X_i$'s
%have a heavy-tailed distribution. We nontheless believe that an answers to the latter conjecture would find their use in the theory of random matrice.

%Tomas nenori minėti 
%Partial progress on Question~\ref{q.2} (and the conjectures) can be made as follows. Use the proof of the balancing lemma, but always keep terms with constant $\pm 1$ equally split in different halves. Then in the end just two types of random variables remain. If the random variables are such that Local Central Limit theorem applies, we get, for example, that if the number of $\pm 1$ coefficients in the initial sum is $n - \Omega(n)$, then $\max_x \pr(\sum a_i X_i = x)$ is smaller than the optimum by a constant factor.
\medskip
\textbf{Acknowledgements.} We would like to thank the reviewers for their useful remarks. We are also grateful to Matas \v Sileikis for a careful reading of our manuscript and numerous corrections that he provided.

\end{document}